\documentclass[12pt,reqno]{amsart}
\usepackage{amsmath}
\usepackage{amssymb}
\usepackage{amstext}
\usepackage{bm}
\usepackage{a4wide}
\usepackage{enumerate}
\allowdisplaybreaks \numberwithin{equation}{section}
\usepackage{color}

\numberwithin{equation}{section}

\newtheorem{theorem}{Theorem}[section]

\theoremstyle{definition}

\theoremstyle{remark}
\newtheorem{remark}[theorem]{Remark}

\begin{document}

\title
{Nonlinear Orbital Stability for Planar Vortex Patches}

 \author{Daomin Cao, Guodong Wang, Jie Wan}
\address{Institute of Applied Mathematics, Chinese Academy of Science, Beijing 100190, and University of Chinese Academy of Sciences, Beijing 100049,  P.R. China}
\email{dmcao@amt.ac.cn}
\address{Institute of Applied Mathematics, Chinese Academy of Science, Beijing 100190, and University of Chinese Academy of Sciences, Beijing 100049,  P.R. China}
\email{wangguodong14@mails.ucas.ac.cn}

\address{Institute of Applied Mathematics, Chinese Academy of Science, Beijing 100190, and University of Chinese Academy of Sciences, Beijing 100049,  P.R. China}
\email{wanjie15@mails.ucas.ac.cn}


\begin{abstract}
In this paper, we prove nonlinear orbital stability for steady vortex patches that maximize the kinetic energy among isovortical rearrangements in a planar bounded domain. As a result, nonlinear stability for an isolated vortex patch is proved. The proof is based on conservation of energy and vorticity, which is an analogue of the classical Liapunov function method.
\end{abstract}

\maketitle
\section{Introduction}
This paper proves that the set of vortex patches as maximizers of the kinetic energy on an isovortical surface(a set of functions with the same distributional function) is orbitally stable for the incompressible Euler equations in a planar bounded domain. Here orbital stability means: if at initial time the flow is close to a maximizer, then it remains close to the set of maximizers. As a consequence of orbital stability, we show stability for isolated maximizers. The key point of the proof is that for an ideal fluid the vorticity moves on an isovortical surface and the kinetic energy is conserved.

In \cite{T}, steady vortex patches were constructed by maximizing the kinetic energy subject to some constraints for vorticity. Burton in \cite{B2,B4} considered more general cases. He constructed various steady vortex flows by maximizing the kinetic energy on rearrangement class, which included the vortex patch solution in \cite{T} as a special case. An interesting and unsolved problem is the stability of these vortex patches. For a single concentrated vortex patch, stability was proved in \cite{CW}, where local uniqueness played an essential role.
But for vortex patches that are not sufficiently concentrated uniqueness is still an open problem, and the method in \cite{CW} does not apply anymore. In this paper, we turn to prove orbital stability for the set of maximizers. The results are stated precisely in Section 2.

For certain domains, there may be no isolated maximizers. For example, for an annular domain, the functional and the constraint are both invariant under rotations, so the set of maximizers is also invariant under rotations. That is the reason we consider orbital stability here. However, if there is an isolated maximizer, we can prove its stability, see Theorem \ref{101} below.

The study of stability for steady Euler flows has a long history. Here we comment on some of the relevant and significant results.  In \cite{K} Kelvin proved linear stability for circular vortex patches in $\mathbb R^2$. Later Love \cite{Lo} proved linear stability for a rotating Kirchhoff elliptical vortex patch. In \cite{A,A2}, Arnold firstly considered nonlinear stability for smooth steady Euler flows, moreover, he came up with the idea that a steady planar Euler flow could be seen as a critical point of the energy on a constraint surface, and stability could be obtained by some kind of non-degenerate condition for this critical point. In 1985, by establishing a relative variational principle for the energy, Wan and Pulvirenti \cite{WP} proved nonlinear stability for circular vortex patches in an open disk. For general bounded domains, Burton in \cite{B3} proved nonlinear stability for steady vortex flows as the strict local maximizer of the energy on rearrangement class. Similar idea was used to prove nonlinear orbital stability for vortex pairs in the whole plane in \cite{B5}. This paper is mostly inspired by \cite{B3} and \cite{B5}.

 The main difficulty in proving orbital stability is to obtain compactness for a particular weakly convergent sequence. In \cite{B5}, compactness was proved by a Concentration-Compactness argument. Here for vortex patches in a bounded domain the proof is relatively simple. In fact, we can prove that any maximizing sequence is compact in $L^p$ norm. The key point is that the weak limit of any maximizing sequence must be a vortex patch, which excludes oscillation and ensures compactness.

Our result also gives a short proof of the stability theorem proved in \cite{B3} for vortex patches and includes the result in \cite{WP} as a special case.

\section{Main Results}

In this section, we state the main result. To begin with, we recall some known facts about the 2-D Euler equations.

Throughout this paper we assume $D$ to be a bounded domain(not necessarily simply-connected) with smooth boundary, $G$ is the Green function for $-\Delta$ in $D$ with zero
boundary condition.

We consider the motion of an ideal fluid in $D$. The governing equations are the following incompressible Euler system

\begin{equation}\label{2}
\begin{cases}
 \nabla\cdot\mathbf{v}=0 \,\,\,\,\,\,\,\,\,\,\,\,\,\,\,\,\,\,\,\,\,\,\,\,\,\,\,\,\,\,\,\,\,\,\,\,\,\,\,\,\,\,\,\,\,\,\text{in $D$},
 \\ \partial_t\mathbf{v}+(\mathbf{v}\cdot\nabla)\mathbf{v}=-\nabla P \,\,\,\,\,\,\,\,\,\,\,\text{in $D$},\\
 \mathbf{v}(x,0)=\mathbf{v}_0(x)\,\,\,\,\,\,\,\,\,\,\,\,\,\,\,\,\,\,\,\,\,\,\,\,\,\,\,\,\,\,\,\text{in $D$},
 \\ \mathbf{v}\cdot \vec{n}=0 \,\,\,\,\,\,\,\,\,\,\,\,\,\,\,\,\,\,\,\,\,\,\,\,\,\,\,\,\,\,\,\,\,\,\,\,\,\,\,\,\,\,\,\,\,\,\,\text{on $\partial D$},
\end{cases}
\end{equation}
where $\mathbf{v}$ is the velocity field, $P$ is the pressure, $\mathbf{v}_0(x)$ is the initial velocity, and $\vec{n}$ is the outward unit normal of $\partial D$. Here we impose the impermeability boundary condition $\mathbf{v}\cdot \vec{n}=0$.

We define the vorticity function $\omega = \partial_1 v_2-\partial_2v_1$. Using the identity $\frac{1}{2}\nabla|\mathbf{v}|^2=(\mathbf{v}\cdot\nabla)\mathbf{v}+J\mathbf{v}\omega$,
 the second equation of $\eqref{2}$ becomes
\begin{equation}\label{3}
 \partial_t\mathbf{v}+\nabla(\frac{1}{2}|\mathbf{v}|^2+P)-J\mathbf{v}\omega=0,
\end{equation}
where $J(v_1,v_2)=(v_2,-v_1)$ denotes clockwise rotation through $\frac{\pi}{2}$. Taking the curl in $\eqref{3}$ gives

\begin{equation}\label{499}
 \partial_t\omega+\mathbf{v}\cdot\nabla\omega=0.
\end{equation}

By the divergence-free condition $\nabla\cdot\mathbf{v}=0$ and the boundary condition $\mathbf{v}\cdot \vec{n}=0$, $\mathbf{v}$ can be written as
\begin{equation}
\mathbf{v}=J\nabla\psi
\end{equation}
 for some function $\psi$ called the stream function(see \cite{MPu}, Chapter 1, Theorem 2.2). Obviously $\psi$ satisfies
\begin{equation}\label{788}
 \begin{cases}
-\Delta \psi=\omega\text{ \quad \,\quad in $D$,}\\
\psi= \text{constant}  \text{\quad on $\partial D$.}
\end{cases}
\end{equation}
Note that for multi-connected domains, $\psi$ is uniquely determined by \eqref{788} provided boundary circulations are prescribed. In this paper, we assume that the stream function vanishes on $\partial D$, i.e.,
\begin{equation}
\psi(x)=\int_DG(x,y)\omega(y)dy.
\end{equation}

Using the notation
 $\partial(\psi,\omega)\triangleq\partial_1\psi\partial_2\omega-\partial_2\psi\partial_1\omega$, $\eqref{499}$ can be written as
\begin{equation}\label{599}
\partial_t\omega+\partial(\omega,\psi)=0.
\end{equation}
Integrating by parts gives the following weak form of $\eqref{599}$:
 \begin{equation}\label{997}
  \int_D\omega(x,0)\xi(x,0)dx+\int_0^{+\infty}\int_D\omega(\partial_t\xi+\partial(\xi,\psi))dxdt=0
  \end{equation}
for all $\xi\in C_0^{\infty}(D\times[0,+\infty))$.

According to Yudovich \cite{Y}, for any initial vorticity $\omega(x,0)\in L^{\infty}(D)$ there is a unique solution to $\eqref{997}$ and
 $\omega(x,t)\in L^{\infty}(D\times(0,+\infty))\cap C([0,+\infty);L^p(D)), \forall \,\,p\in [1,+\infty)$. Moreover, $\omega(x,t)\in R_{\omega_0}$ for all $t\geq 0$. Here $R_{\omega}$ denotes the rearrangement class of a given function $\omega$, that is,
 \begin{equation}
 R_\omega\triangleq\{ v |  |\{v>a\}|=|\{\omega>a\}| ,\forall a\in \mathbb R^1\}.
 \end{equation}
where $|A|$ denotes the area of a set $A\subset\mathbb R^2$. For convenience, we also write $\omega(x,t)$ as $\omega_t(x)$.

If $\omega$ is a solution of \eqref{997} and is independent of $t$, it is called steady. In this paper, we consider steady vortex patch solution having the form $\omega=\lambda I_A$, where $I_{A}$ denotes the characteristic function of some measurable set $A$, i.e., $I_A(x)=1$ in $A$ and $I_A=0$ elsewhere.  It is easy to see that $\omega$ is steady if and only if
\begin{equation}\label{888}
\int_D\omega\partial(\xi,\psi) dx=0,\text{ for all $\xi\in C^\infty_0(D).$}
\end{equation}

There are many ways to construct steady vortex patches. Here we consider the construction in \cite{T}. Define
 \[K\triangleq\{\omega\in L^{\infty}(D)|\,\,0\leq \omega\leq \lambda,\,\int_{D}\omega(x)dx=1\},\]
 where $\lambda$ is any given positive constant.  For any $\omega\in K$, the kinetic energy of $\omega$ is defined by
\begin{equation}
E(\omega)\triangleq\frac{1}{2}\int_D \int_D G(x,y)\omega(x)\omega(y)dxdy.
\end{equation}

For planar vortex flow the kinetic energy is conserved, i.e., if $\omega_t\in L^\infty(D)$ is a solution to \eqref{997} with initial vorticity $\omega_0$, then $E(\omega_t)=E(\omega_0)$ for all $t\geq0$, see Theorem 14 in \cite{B3} for a detailed proof. In this paper, we use $E$ as the Liapunov function for the Euler dynamical system to obtain stability.

\begin{theorem}[Turkington, \cite{T}]\label{51}
  $E$ attains its maximum on $K$ and any maximizer is a steady vortex patch.
\end{theorem}
For reader's convenience and completeness, we give the proof of Theorem \ref{51} in Section 3.

In the sequel, we denote $M$ the set of all maximizers, and define
\begin{equation}
R\triangleq\{\omega|\omega=\lambda I_A, \lambda|A|=1\}.
\end{equation}
By Theorem \ref{51}, any maximizer is a vortex patch, so $M\subset R$.
\begin{remark}

It is easy to verify that $sup_{\omega\in K}E(\omega)=sup_{\omega\in R}E(\omega)$, which means that $M$ is in fact the set of maximizers of $E$ on rearrangement class $R$. The variational problem on rearrangement class has been considered by Burton, see \cite{B2}, \cite{B4} for example.
\end{remark}

Our purpose in this paper is to prove the orbital stability of $M$.
\begin{theorem}\label{666}
$M$ is orbitally stable. More specifically, for any $\varepsilon>0$, there exists $\delta>0$, such that for any $\omega_0\in K$, $dist(\omega_0,M)<\delta$, we have $dist(\omega_t,M)<\varepsilon$ for all $t\geq0$. Here the distance is in the sense of $L^p$ norm for any $1\leq p<+\infty$, $\omega_t$ is the solution to \eqref{997} with initial vorticity $\omega_0$.
\end{theorem}

\begin{remark}
In \cite{CW,MPu,Ta}, the perturbed vorticity $\omega_0$ is restricted on the isovortical surface $R$, here we extend the perturbation set to $K$.
\end{remark}

For certain domains, the maximizer of $E$ may be not isolated in $L^p(D)$. For example, when $D$ is a ring, i.e., $D=B_R(x_0)\verb|\|B_r(x_0)$ for some $x_0\in \mathbb R^2, 0<r<R<+\infty$, the rotation of any maximizer is still a maximizer. But once there is an isolated maximizer, we can prove stability.

\begin{theorem}\label{101}
Assume that $\omega_\lambda$ is an isolated maximizer of $E$ on $K$, i.e., there exists some $\delta_0>0$ such that for any $\omega\in K$, $0<dist(\omega,\omega_\lambda)<\delta_0$, we have $E(\omega)<E(\omega_\lambda)$, then $\omega_\lambda$ is stable. More specifically, for any $\varepsilon>0$, there exists $\delta>0$, such that for any $\omega_0\in K$, $dist(\omega_\lambda,\omega_0)<\delta$, we have $dist(\omega_\lambda,\omega_t)<\varepsilon$ for all $t\geq0$. Here the distance is in the sense of $L^p$ norm for any $1\leq p<+\infty$, $\omega_t$ is the solution to \eqref{997} with initial vorticity $\omega_0$.
\end{theorem}

\begin{remark}
In \cite{B3}, a more general stability theorem for isolated maximizers was proved, here we give a different and short proof in the case of vortex patches.
\end{remark}

\begin{remark}
In some cases the maximizer is unique and thus isolated. For example, when $D$ is a convex domain, there is a unique maximizer provided $\lambda$ is large enough, see \cite{CGPY}, or when $D$ is an open disc, for each $\lambda$ there is a unique maximizer, namely the circular vortex patch concentric to $D$ with radius $\frac{1}{\sqrt{\lambda\pi}}$, see \cite{BM}, Theorem 3.1.
\end{remark}

\section{Proofs}

In this section we prove the main results.

\begin{proof}[Proof of Theorem \ref{51}]

Step 1: $E$ attains its maximum. Notice that $G(x,y)\in L^{1}(D\times D)$, thus for any $\omega\in K$,
\[E(\omega)=\frac{1}{2}\int_D\int_DG(x,y)\omega(x)\omega(y)dxdy\leq \frac{1}{2}\lambda^2\int_D\int_D|G(x,y)|dxdy\leq C\lambda^2,
\]
where $C$ is a positive number depending on $D$, which means that $E$ is bounded from above on $K$. Let $\{\omega^n\}\subset K$ be a maximizing sequence. Since $K$ is bounded in $\L^\infty(D)$, $K$ is sequentially compact in the weak star topology in $L^\infty(D)$. Without loss of generality we assume that $\omega^n\rightarrow \omega^*$ weakly star in $L^\infty(D)$ for some $\omega^*\in L^\infty(D)$ as $n\rightarrow +\infty$.

We claim that $\omega^*\in K$. In fact, $\omega^n\rightarrow \omega^*$ weakly star in $L^\infty(D)$ means
\[\lim_{n\rightarrow +\infty}\int_D\omega^n\phi=\int_D\omega^*\phi\]
for any $\phi\in L^1(D)$. Choosing $\phi\equiv1$, we have
\[\lim_{n\rightarrow +\infty}\int_D\omega^n=\int_D\omega^*=1.\]
Now we prove $0\leq \omega^*\leq\lambda$ by contradiction. Suppose that $|\{\omega^*>\lambda\}|>0$, then there exists $\varepsilon_0>0$ such that $|\{\omega^*\geq\lambda+\varepsilon_0\}|>0$. Denote $A=\{\omega^*\geq\lambda+\varepsilon_0\}$, then for $\phi=I_A$ we have
\[0=\lim_{n\rightarrow +\infty}\int_D(\omega^*-\omega^n)\phi=\lim_{n\rightarrow +\infty}\int_{A}\omega^*-\omega^n.\]
On the other hand
\[\lim_{n\rightarrow +\infty}\int_{A}\omega^*-\omega^n\geq\varepsilon_0|A|>0,\]
which is a contradiction. So we have $\omega^*\leq \lambda$. Similarly we can prove $\omega^*\geq 0$.

Finally since $G(x,y)\in L^1(D\times D)$, we have $\lim_{n\rightarrow +\infty} E(\omega^n)=E(\omega^*)$, so $\omega^*$ is a maximizer of $E$.

Step 2: Any maximizer satisfies \eqref{888}. For any $\xi\in C^{\infty}_0(D)$, we define a family of transformations $\Phi_t(x)$ from $D$ to $D$ by the following equations,
\begin{equation}\label{400}
\begin{cases}\frac{d\Phi_t(x)}{dt}=J\nabla\xi(\Phi_t(x)),\,\,\,t\in\mathbb R^1, \\
\Phi_0(x)=x.
\end{cases}
\end{equation}
Since $J\nabla\xi$ is a smooth vector field with compact support, $\eqref{400}$ is solvable for all $t\in \mathbb R^1$. It is easy to verify that $J\nabla\xi$ is divergence-free, so by Liouville theorem(see \cite{MPu}, Appendix 1.1) $\Phi_t$ is area-preserving. Let $\omega^*$ be any maximizer and define a family of test functions
\begin{equation}
\omega^t(x)\triangleq\omega^*(\Phi_t(x)).
\end{equation}
Obviously $\omega^t\in K$, so $\frac{dE(\omega^t)}{dt}|_{t=0}=0$.
 Expanding $E(\omega^t)$ at $t=0$ gives
\[\begin{split}
E(\omega^t)=&\frac{1}{2}\int_D\int_DG(x,y)\omega^*(\Phi_t(x))\omega^*(\Phi_t(y))dxdy\\
=&\frac{1}{2}\int_D\int_DG(\Phi_{-t}(x),\Phi_{-t}(y))\omega^*(x)\omega^*(y)dxdy\\
=&E(\omega^*)+t\int_D\omega^*\partial(\psi^*,\xi)+o(t),
\end{split}\]
as $t\rightarrow 0$, where $\psi^*$ is the stream function. So we have
\[\int_D\omega^*\partial(\psi^*,\xi)=0.\]

 Step 3: Any maximizer is a vortex patch. Let $\omega^*$ be any maximizer and define a family of test functions $\omega^s(x)=\omega^*+s[z_0(x)-z_1(x)]$, $s>0$, where $z_0,z_1$ satisfies
\begin{equation}
\begin{cases}
z_0,z_1\in L^\infty(D),

 \\ \int_Dz_0=\int_D z_1,
 \\ z_0,z_1\geq 0,
 \\z_0=0\text{\,\,\,\,\,\,} in\text{\,\,} D\verb|\|\{\omega^*\leq\lambda-\delta\},
 \\z_1=0\text{\,\,\,\,\,\,} in\text{\,\,} D\verb|\|\{\omega^*\geq\delta\}.
\end{cases}
\end{equation}
Here $\delta$ is any positive number. Note that for fixed $z_0,z_1$ and $\delta$, $\omega^s\in K$ provided $s$ is sufficiently small. So we have
\[0\geq\frac{dE(\omega^s)}{ds}|_{s=0^+}=\int_Dz_0\psi^*-\int_Dz_1\psi^*,\]
 which gives
\[\sup_{\{\omega^*<\lambda\}}\psi^*\leq\inf_{\{\omega^*>0\}}\psi^*.\]
Since $D$ is connected and $\overline{\{\omega^*<\lambda\}}\cup\overline{\{\omega^*>0\}}=D$, we have $\overline{\{\omega^*<\lambda\}}\cap\overline{\{\omega^*>0\}}\neq\varnothing$, then by continuity of $\psi^*$,
\[\sup_{\{\omega^*<\lambda\}}\psi^*=\inf_{\{\omega^*>0\}}\psi^*.\]
Now define \[\mu\triangleq\sup_{\{\omega^*<\lambda\}}\psi^*=\inf_{\{\omega^*>0\}}\psi^*,\] we have
\begin{equation}
\begin{cases}
\omega^*=0\text{\,\,\,\,\,\,$a.e.$\,} in\text{\,\,}\{\psi^*<\mu\},
 \\ \omega^*=\lambda\text{\,\,\,\,\,\,$a.e.$\,} in\text{\,\,}\{\psi^*>\mu\}.
\end{cases}
\end{equation}
On $\{\psi^*=\mu\}$, we have $\nabla\psi^*=0\text{\,\,} a.e.$, which gives $\omega^*=-\triangle \psi^*=0$. That is,
\begin{equation}
\begin{cases}
\omega^*=0\text{\,\,\,\,\,\,$a.e.$\,} in\text{\,\,}\{\psi^*\leq\mu\},
 \\ \omega^*=\lambda\text{\,\,\,\,\,\,$a.e.$\,} in\text{\,\,}\{\psi^*>\mu\},
\end{cases}
\end{equation}
 so $\omega^*$ is a vortex patch.

\end{proof}

Now we turn to the proof of Theorem \ref{666}. The key point is compactness. Generally speaking, for a weak convergent function sequence in $K$, strong convergence may fail because of oscillation, but here for a maximizing sequence we can prove that the weak convergence limit is a vortex patch, which will be used to exclude oscillation and obtain compactness.

In the sequel, $p\in[1,+\infty)$ is fixed, and $|f|_p$ denotes the $L^p$ norm of some function $f$.
\begin{proof}[Proof of Theorem \ref{666}]
We prove by contradiction in the following.

Suppose that there exist $\varepsilon_0>0$, $\{\omega^n_0\}\subset K$, $\{t^n\}\subset \mathbb{R}^+$ such that
\begin{equation}\label{10}
dist(\omega^n_0,M)\rightarrow 0,
\end{equation}
and
\begin{equation}\label{11}
dist(\omega^n_{t_n},M)\geq\varepsilon_0,
\end{equation}
for any $n$, where $\omega^n_{t_n}$ is the solution to \eqref{997} at time $t_n$ with initial vorticity $\omega^n_0$. By vorticity conservation(see \cite{MPu}, Chapter 1) $\omega^n_{t_n}$ has the same distributional function as $\omega^n_0$(or $\omega^n_{t_n}\in R_{\omega^n_0}$), so $\omega^n_{t_n}\in K$.

From \eqref{10}, we can choose$\{v^n\}\subset M$ such that
\begin{equation}
|\omega^n_0-v^n|_p\rightarrow 0.
\end{equation}
We claim that $\{\omega^n_0\}$ is an energy maximizing sequence for $E$ on $R$. In fact,
\begin{equation}\label{19}
\begin{split}
E(\omega^n_0)-E(v^n)=&\frac{1}{2}\int_D\int_DG(x,y)\left[\omega^n_0(x)\omega^n_0(y)-v^n(x)v^n(y)\right]\\
=&\frac{1}{2}\int_D\int_DG(x,y)\left[\omega^n_0(x)\omega^n_0(y)-v^n(x)\omega^n_0(y)+v^n(x)\omega^n_0(y)-v^n(x)v^n(y)\right]\\
=&\frac{1}{2}\int_D\int_DG(x,y)\omega^n_0(y)\left[\omega^n_0(x)-v^n(x)\right]+\frac{1}{2}\int_D\int_DG(x,y)v^n(x)\left[\omega^n_0(y)-v^n(y)\right]\\
=&\frac{1}{2}\int_D\xi^n(x)\left[\omega^n_0(x)-v^n(x)\right]+\frac{1}{2}\int_D\zeta^n(y)\left[\omega^n_0(y)-v^n(y)\right]
\end{split}
\end{equation}
 where $\xi^n(x)=\int_DG(x,y)\omega^n_0(y),\zeta^n(y)=\int_DG(x,y)v^n(x)$. Since $\{\omega^n_0\},\{v^n\}$ are both bounded in $L^{\infty}(D)$, by $L^p$ estimates $\{\xi^n\}$, $\{\zeta^n\}$ are bounded in $W^{2,r}(D)$ for any $r\in[1,+\infty)$ and thus bounded in $L^{\infty}(D)$. Combining \eqref{19} we have \begin{equation}E(\omega^n_0)=E(v^n)+o(1)\end{equation} as $n\rightarrow +\infty$, which means that $\{\omega^n_0\}$ is an energy maximizing sequence.

 By energy conservation we have
\begin{equation}
E(\omega^n_0)=E(\omega^n_{t_n}),
\end{equation}
so $\{\omega^n_{t_n}\}$ is also an energy maximizing sequence.
For convenience, we write $u_n\triangleq \omega^n_{t_n}.$ Now choose $q$ to be fixed, $1\leq p<q<+\infty$, since $u_n\in K$, we know that $\{u_n\}$ is a bounded sequence in $L^q(D)$.  Without loss of generality, we assume that $u_n\rightarrow u$ weakly in $L^q(D)$.

$Claim:$ $u\in K$ and $u$ is an energy maximizer of $E$ on $K$.

$Proof\,\, of\,\, the\,\, Claim:$ Firstly, $u_n\rightarrow u$ weakly in $L^q(D)$ implies
\[\lim_{n\rightarrow +\infty}\int_Du_n\phi=\int_Du\phi\]
for any $\phi\in L^{q^*}(D)$, where $q^*=\frac{q}{q-1}$. By choosing $\phi\equiv 1$ we have
\[1=\lim_{n\rightarrow +\infty}\int_Du_n=\int_Du.\]
 Now we prove $u\leq \lambda$ by contradiction. Suppose that $|\{u>\lambda\}|>0$, then there exists $\varepsilon_1>0$ such that $|\{u>\lambda+\varepsilon_1\}|>0$. Denote $A=\{u>\lambda+\varepsilon_1\}$, then for any $\phi=I_{A}$ weak convergence implies
\[\lim_{n\rightarrow +\infty}\int_D(u-u_n)\phi=0,\]
but on the other hand
\[\lim_{n\rightarrow +\infty}\int_D(u-u_n)\phi=\int_{A}u-u_n\geq|A|\varepsilon_1>0,\]
which is a contradiction. Similar argument gives $ u\geq 0$. Finally, since $G\in L^1(D\times D)$, we have $\lim_{n\rightarrow+\infty}E(u_n)=E(u)$, which means $u$ is an energy maximizer on $K$. Thus the claim is proved.

From the claim $u\in M$, thus by \eqref{11}
 \begin{equation}\label{20}
 |u-u_n|_p\geq\varepsilon_0
 \end{equation}
 for any $n$.

 According to Theorem \ref{51}, any maximizer of $E$ on $K$ must be a vortex patch, so
 \begin{equation}
 \int_D |u|^q=\lambda^{q-1}.
 \end{equation}
 Now we show that $\lim_{n\rightarrow+\infty}\int_D|u_n|^q= \int_D|u|^q.$ On the other hand, by weak lower semi-continuity of $L^q$ norm
 \begin{equation}
  \lambda^{q-1}=\int_D|u|^q\leq\varliminf_{n\rightarrow+\infty}\int_D|u_n|^q=\varliminf_{n\rightarrow+\infty}\int_D|\omega^n_0|^q,
 \end{equation}
on the other hand, $\omega^n_0\in K$ gives
  \begin{equation}
\int_D|\omega^n_0|^q=\int_D|\omega^n_0|^{q-1}\omega^n_0\leq\lambda^{q-1}\int_D\omega^n_0=\lambda^{q-1},
 \end{equation}
 so
 \begin{equation}\label{61}
 \lim_{n\rightarrow+\infty}\int_D|u_n|^q= \int_D|u|^q.
 \end{equation}

 That is, $u_n\rightarrow u$ weakly in $L^q(D)$ and $\int_D |u_n|^q\rightarrow\int_D |u|^q$, then immediately we have $u_n\rightarrow u$ in $L^q(D)$ by uniform convexity of $L^q$ norm(recall $q$ is chosen such that $1\leq p<q<+\infty$). By H\"{o}lder inequality we have $u_n\rightarrow u$ in $L^p(D)$, which is a contradiction to \eqref{20}.

\end{proof}

\begin{proof}[Proof of Theorem \ref{101}]
Denote $N=M\verb|\|\{\omega_\lambda\},$ then $dist(\omega_\lambda,N)\geq\delta_0.$ By orbital stability, for any $\varepsilon$, $0<\varepsilon<\frac{\delta_0}{4}$, there exists $\delta>0$, $\delta<\frac{\delta_0}{2}$, such that for any $\omega_0\in K$, $dist(\omega_0,\omega_\lambda)<\delta$, we have $dist(\omega_t, M)<\varepsilon$ for all $t\geq0$. We have

\begin{equation}\label{81}
\min\{dist(\omega_t, \omega_\lambda),dist(\omega_t, N)\}\leq \varepsilon
\end{equation}
 for all $t\geq0$. We claim that
 \begin{equation}\label{85}
 dist(\omega_t, N)>\varepsilon
 \end{equation}
for all $t\geq0$. In fact, suppose that there is $t_1\geq0$ such that $dist(\omega_{t_1}, N)\leq\varepsilon $, then
\begin{equation}
\varepsilon\geq dist(\omega_{t_1}, N)\geq dist(\omega_\lambda,N)-dist(\omega_\lambda,\omega_{t_1})\geq \delta_0-dist(\omega_\lambda,\omega_{t_1}),
\end{equation}
since $\varepsilon<\frac{\delta_0}{4}$, we have
\begin{equation}
dist(\omega_{t_1},\omega_\lambda)>\frac{3}{4}\delta_0.
\end{equation}
That is, $dist(\omega_0,\omega_\lambda)<\delta<\frac{\delta_0}{2}$ and $dist(\omega_{t_1},\omega_\lambda)>\frac{3}{4}\delta_0$, by continuity(recall that $\omega_t\in C([0,+\infty);L^p(D))$ for all $p\in[1,+\infty)$) there exists $t_2$ such that
\begin{equation}\label{82}
dist(\omega_{t_2},\omega_\lambda)=\frac{\delta_0}{2}>\varepsilon,
\end{equation}
thus

\begin{equation}\label{83}
dist(\omega_{t_2},N)\geq dist(\omega_\lambda,N)-dist(\omega_{t_2},\omega_\lambda)\geq \frac{\delta_0}{2}>\varepsilon.
\end{equation}
Combing \eqref{81},\eqref{82} and \eqref{83}, we get a contradiction. Now \eqref{81} and \eqref{85} give

\begin{equation}
dist(\omega_t, \omega_\lambda)\leq \varepsilon
\end{equation}
 for all $t\geq0$ provided $dist(\omega_0,\omega_\lambda)<\delta$, which is the desired result.
\end{proof}

\end{document}